\documentclass[reqno]{amsart}
\usepackage{amsmath, amsfonts, amsthm, amssymb, setspace, textcomp, bbm, multirow}
\usepackage{geometry}
\newtheorem{theorem}{Theorem}[section]
\newtheorem{lemma}[theorem]{Lemma}

\theoremstyle{definition}

\theoremstyle{remark}

\numberwithin{equation}{section}

\newcommand{\spt}[2]{\mbox{\normalfont spt}_{#1}\Parans{#2}}
\newcommand{\sptBar}[2]{\overline{\mbox{\normalfont spt}}_{#1}\Parans{#2}}

\newcommand{\Parans}[1]{\left(#1\right)}

\newcommand{\PieceTwo}[4]
{
	\left\{
   	\begin{array}{ll}
      	#1 & #3 \\
       	#2 & #4
     	\end{array}
	\right.
}
\newcommand{\aqprod}[3]{\Parans{#1;#2}_{#3}}

\newcommand{\Jac}[2]{\left(\frac{#1}{#2}\right)}

\author{CHRIS JENNINGS-SHAFFER}
\address{Department of Mathematics, Oregon State University\\
Corvallis, Oregon 97331, USA
\endgraf jennichr@math.oregonstate.edu}



\title{Some Smallest Parts Functions from Variations of Bailey's Lemma }

\allowdisplaybreaks
\begin{document}

\allowdisplaybreaks

\begin{abstract}
We construct new smallest parts partition functions and smallest parts crank functions
by considering variations of Bailey's Lemma and conjugate Bailey pairs. The
functions we introduce satisfy simple linear congruences modulo $3$ and $5$.
We introduce and give identities for two four variable $q$-hypergeometric functions; 
these functions specialize to some of our new spt-crank-type functions as well as 
many known spt-crank-type functions.
\end{abstract}

\maketitle

\section{Introduction}
\allowdisplaybreaks

In the recent study of ranks and cranks for smallest parts partition it has 
become apparent that Bailey pairs and Bailey's Lemma are inherent to this 
study. One can review the articles 
\cite{AGL, GarvanJennings, GarvanJennings2, Patkowski} 
to see this is the case.
We will demonstrate another method in which this occurs.
We recall a partition of an integer $n$ is a non-increasing sequence of 
positive integers that
sum to $n$. We let $p(n)$ denote the number of partitions of $n$; as an example
$p(5)=7$ since the partitions of $5$ are 
$5$, $4+1$, $3+2$, $3+1+1$, $2+2+1$, $2+1+1+1$, and $1+1+1+1+1$.
There are many functions related to $p(n)$, one of them is $\spt{}{n}$, 
the smallest parts partition function.
The smallest parts partition function was 
introduced by Andrews in \cite{Andrews} as a weighted count on the partitions 
of $n$, by counting each partition by the number of times the smallest
part appears. 
From the partitions of $5$ we see that $\spt{}{5}=14$. 

Both $p(n)$ and $\spt{}{n}$ satisfy certain linear congruences, as do many
functions related to partitions. In particular the partition function satisfies
$p(5n+4)\equiv 0\pmod{5}$, $p(7n+5)\equiv 0\pmod{7}$, and 
$p(11n+6)\equiv 0\pmod{11}$ and the smallest parts function satisfies
$\spt{}{5n+4}\equiv 0\pmod{5}$, $\spt{}{7n+5}\equiv 0\pmod{7}$, and
$\spt{}{13n+6}\equiv 0\pmod{13}$. The congruences for $p(n)$ were first 
observed by Ramanujan and now have a plethora of proofs, a few of which can be
found in \cite{AS, Berndt2, Garvan1, Ramanujan}. The congruences for 
$\spt{}{n}$ were first established by
Andrews when he introduced the function in \cite{Andrews}. Focusing on 
$\spt{}{n}$, we note that Andrews original congruences were met with a storm of
articles establishing various facts about the smallest parts function and new 
congruences. To name just a few of these congruences, in \cite{Ono2} Ono established for
$\ell\ge 5$ prime and $\Jac{-\delta}{\ell}=1$ that
\begin{align*}
	\spt{}{\frac{\ell^2(\ell n+\delta)+1}{24}  }\equiv 0\pmod{\ell}
	,
\end{align*}
in \cite{Garvan3} Garvan established for $a,b,c\ge 1$ and with
$\delta_a$, $\lambda_b$, $\gamma_c$ respectively denoting the least non-negative 
reciprocals of $24$ modulo  $5^a$, $7^b$, and $13^c$ that
\begin{align*}
	\spt{}{5^an+\delta_a}&\equiv 0 \pmod{5^{ \left\lfloor\frac{a+1}{2}\right\rfloor  }}
	,\\
	\spt{}{7^bn+\lambda_b}&\equiv 0 \pmod{7^{ \left\lfloor\frac{b+1}{2}\right\rfloor  }}
	,\\
	\spt{}{13^cn+\gamma_c}&\equiv 0 \pmod{13^{ \left\lfloor\frac{c+1}{2}\right\rfloor  }}
	,
\end{align*}
and in \cite{AhlgrenBringmannLovejoy} Ahlgren, Bringmann, and Lovejoy established
for $\ell\ge 5$ prime, $m\ge 1$, and $\Jac{-n}{\ell}=1$ that
\begin{align*}
	\spt{}{\frac{\ell^2n+1}{24}  }\equiv 0\pmod{\ell^m}
	.
\end{align*}

The smallest parts function is of interest not just for its congruences and
elegant combinatorial description, but also for its modular properties.
While the generating function for $p(n)$ is a modular form of weight
$-1/2$, the generating function for $\spt{}{n}$ is instead one of two pieces
of the so-called holomorphic part of a weight $3/2$ harmonic Maas form
\cite{Bringmann,Ono2}. Maass forms are of great recent interest in number
theory. One application of Maass forms is that they give new explanations of
Ramanujan's mock theta conjectures \cite{Andersen1, Folsom}. Other smallest parts functions
were also found to be related to Maass forms in \cite{BLO1}.

Here we are interested in studying a wide array of functions whose generating
functions have a similar form to that of $\spt{}{n}$. We further restrict our
attention to those that satisfy simple linear congruences, 
like $\spt{}{5n+4}\equiv 0\pmod{5}$, that can be explained
by a so-called spt-crank. The first spt-crank was given by Andrews, Garvan, and
Liang in \cite{AGL}, however the idea of using partition statistics to explain
partition congruences originated with Dyson's rank conjectures 
\cite{AS, Dyson}. This topic experienced a resurgence beginning with Garvan's vector 
crank \cite{Garvan1} and the Andrews-Garvan crank \cite{AndrewsGarvan}.
The idea of a rank or crank is to define a statistic that yields a refinement of
a congruence. This is best illustrated with an example. The Dyson rank of a 
partition is defined as the largest part minus the number of parts. The 
partition function satisfies the congruence $p(5n+4)\equiv 0\pmod{5}$. If one
groups the partitions of $5n+4$ according to the value of their rank reduced 
modulo $5$, then one has five sets of equal size. This can be seen with the
partitions of $4$ in the following table.
\begin{align*}
\begin{array}{c|c|c}
	partition & rank & rank \pmod{5}
	\\
	\hline
	4& 3 & 3
	\\
	3+1& 1 & 1
	\\
	2+2& 0 & 0
	\\	
	2+1+1& -1 & 4
	\\	
	1+1+1+1& -3 & 2
\end{array}
\end{align*}
For a given partition like function with a congruence, one can attempt to find
a statistic to explain the congruence. However, the statistic may not have 
such an elegant and simple definition as the Dyson rank of a partition.

The goal of this article is to demonstrate a method by which we can find and 
introduce new partitions functions while simultaneously obtaining a crank 
function for them. From this method we select those functions
that satisfy simple linear congruences and by which the crank gives a 
proof of the congruences. In Section 2 we introduce these new functions,
state various identities and congruences, and prove a few preliminary 
results. In Section 3 we describe with an example our method for finding these
new functions. In Sections 4 and 5 we prove the results stated in Section 2. Finally 
in Section 6 we give a few concluding remarks.

\section{Preliminaries and Statement of Results}

Throughout this article we use the standard product notation,
\begin{align*}
	\aqprod{z}{q}{n} 
		&= \prod_{j=0}^{n-1} (1-zq^j)
	,
	&\aqprod{z}{q}{\infty} 
		&= \prod_{j=0}^\infty (1-zq^j)
	,\\
	\aqprod{z_1,\dots,z_k}{q}{n} 
		&= \aqprod{z_1}{q}{n}\dots\aqprod{z_k}{q}{n}
	,
	&\aqprod{z_1,\dots,z_k}{q}{\infty} 
		&= \aqprod{z_1}{q}{\infty}\dots\aqprod{z_k}{q}{\infty}
.
\end{align*}
To begin we define we define two generic functions,
\begin{align*}
	F(\rho_1,\rho_2,z;q)
	&=
	\frac{\aqprod{q}{q}{\infty}}{\aqprod{z,z^{-1},\rho_1,\rho_2}{q}{\infty}}
	\sum_{n=1}^\infty
	\frac{\aqprod{z,z^{-1},\rho_1,\rho_2}{q}{n}(\tfrac{q}{\rho_1\rho_2})^n}{\aqprod{q}{q}{2n}}	
	,\\	
	G(\rho_1,\rho_2,z;q)
	&=
	\frac{\aqprod{q}{q}{\infty}}{\aqprod{z,z^{-1},\rho_1,\rho_2}{q}{\infty}}
	\sum_{n=1}^\infty
	\frac{\aqprod{z,z^{-1},\rho_1,\rho_2}{q}{n}(\tfrac{q^2}{\rho_1\rho_2})^n}{\aqprod{q}{q}{2n}}	
	.
\end{align*}
We would also like to let $\rho_2\rightarrow\infty$ in $F(\rho_1,\rho_2,z;q)$
and $G(\rho_1,\rho_2,z;q)$, however this requires a slight alteration. In
particular we let
\begin{align*}
	F(\rho,z;q)
	&=
	\lim_{\rho_2\rightarrow\infty}\aqprod{\rho_2}{q}{\infty}F(\rho,\rho_2,z,q)
	=
	\frac{\aqprod{q}{q}{\infty}}{\aqprod{\rho,z,z^{-1}}{q}{\infty}}
	\sum_{n=1}^\infty\frac{\aqprod{z,z^{-1},\rho}{q}{n}(-1)^nq^{\frac{n(n+1)}{2}}\rho^{-n}}{\aqprod{q}{q}{2n}}
	,\\
	G(\rho,z;q)
	&=
	\lim_{\rho_2\rightarrow\infty}\aqprod{\rho_2}{q}{\infty}G(\rho,\rho_2,z,q)
	=
	\frac{\aqprod{q}{q}{\infty}}{\aqprod{\rho,z,z^{-1}}{q}{\infty}}
	\sum_{n=1}^\infty\frac{\aqprod{z,z^{-1},\rho}{q}{n}(-1)^nq^{\frac{n(n+3)}{2}}\rho^{-n}}{\aqprod{q}{q}{2n}}
.
\end{align*}
The special cases of these functions we are interested in are
\begin{align*}
	S_{G1}(z,q) 
	&= 
		G(q,q^2,z;q^2)
		=
		\frac{\aqprod{q^2}{q^2}{\infty}}{\aqprod{z,z^{-1},q,q^2}{q^2}{\infty}}
		\sum_{n=1}^\infty
		\frac{\aqprod{z,z^{-1},q,q^2}{q^2}{n}q^n}{\aqprod{q^2}{q^2}{2n}}
	,\\
	S_{G2}(z,q)	
		&=
		G(iq^{1/2},-iq^{1/2},z;q)
		=
		\frac{\aqprod{q}{q}{\infty}}{\aqprod{z,z^{-1}}{q}{\infty}\aqprod{-q}{q^2}{\infty}}
		\sum_{n=1}^\infty
		\frac{\aqprod{z,z^{-1}}{q}{n}\aqprod{-q}{q^2}{n}q^n}
			{\aqprod{q}{q}{2n}}
	,\\
	S_{F1}(z,q)
		&=
		F(-q,z;q)
		=
		\frac{\aqprod{q}{q}{\infty}}{\aqprod{z,z^{-1},-q}{q}{\infty}}
		\sum_{n=1}^\infty
		\frac{\aqprod{z,z^{-1},-q}{q}{n}q^{\frac{n(n-1)}{2}}}
			{\aqprod{q}{q}{2n}}	
	,\\
	S_{G3}(z,q) 
		&=
		G(q,z;q)
		=
		\frac{\aqprod{q}{q}{\infty}}{\aqprod{z,z^{-1},q}{q}{\infty}}
		\sum_{n=1}^\infty
		\frac{\aqprod{z,z^{-1},q}{q}{n}(-1)^nq^{\frac{n(n+1)}{2}}}
			{\aqprod{q}{q}{2n}}	
	.
\end{align*}
Additionally we define three functions of a similar form 
that we will find are related to the conjugate Bailey
pair identities $(1.7)$, $(1.9)$, and $(1.12)$ of \cite{Lovejoy},
\begin{align*}
	S_{L7}(z;q)
	&=
		\frac{\aqprod{-q}{q}{\infty}}{\aqprod{z,z^{-1}}{q^2}{\infty}}
		\sum_{n=1}^\infty
		\frac{\aqprod{z,z^{-1}}{q^2}{n}q^{2n}}{\aqprod{-q}{q}{2n}}	
	,\\
	S_{L9}(z;q)
	&=	
		\frac{\aqprod{q}{q^2}{\infty}}{\aqprod{z,z^{-1}}{q}{\infty}}
		\sum_{n=1}^\infty
		\frac{\aqprod{z,z^{-1}}{q}{n}q^{n}}{\aqprod{q}{q^2}{n}}		
	,\\
	S_{L12}(z;q)
	&=
		\frac{\aqprod{-q}{q^2}{\infty}^2}{\aqprod{z,z^{-1}}{q^2}{\infty}}
		\sum_{n=1}^\infty
		\frac{\aqprod{z,z^{-1}}{q^2}{n}q^{2n}}{\aqprod{-q}{q^2}{n}\aqprod{-q}{q^2}{n+1}}	
	.
\end{align*}
These seven functions are our spt-crank functions.
By setting $z=1$ and simplifying the products, we obtain our smallest parts 
functions.
\begin{align*}
	S_{G1}(q)
	=
	\sum_{n=1}^\infty \spt{G1}{n}q^n	
		&=
		\sum_{n=1}^\infty
		\frac{q^n\aqprod{q^{4n+2}}{q^2}{\infty}}{\aqprod{q^{2n},q^{2n},q^{2n+1},q^{2n+2}}{q^2}{\infty}}
		\\&\quad		
		=
		\sum_{n=1}^\infty
		\frac{q^n}{(1-q^{2n})^2\aqprod{q^{2n+1}}{q}{\infty}}	
		\cdot\frac{1}{\aqprod{q^{2n+2}}{q^2}{\infty}}
		\cdot\frac{1}{\aqprod{q^{2n+2}}{q^2}{n}}		
	,\\
	S_{G2}(q)
	=
		\sum_{n=1}^\infty \spt{G2}{n}q^n	
		&=
		\sum_{n=1}^\infty
		\frac{q^n\aqprod{q^{2n+1}}{q}{\infty}}{\aqprod{q^n}{q}{\infty}^2\aqprod{-q^{2n+1}}{q^2}{\infty}}
		\\		
		&=
		\sum_{n=1}^\infty
		\frac{q^n}{(1-q^n)^2\aqprod{q^{n+1}}{q}{n}\aqprod{q^{2n+2}}{q^2}{\infty}}			
		\cdot		
		\frac{1}{\aqprod{q^{n+1}}{q}{n}\aqprod{q^{4n+2}}{q^4}{\infty}}
	,\\
	S_{F1}(q)
	=
		\sum_{n=1}^\infty \spt{F1}{n}q^n	
		&=
		\sum_{n=1}^\infty
		\frac{q^{\frac{n(n-1)}{2}}\aqprod{q^{2n+1}}{q}{\infty}}{\aqprod{q^n,q^n,-q^{n+1}}{q}{\infty}}
		\\
		&=
		\frac{1}{(1-q)^2\aqprod{q^2}{q^2}{\infty}}		
		+
		\frac{q}{(1-q^2)^2(1-q^3)\aqprod{q^4}{q^2}{\infty}}		
			\\&\quad
			+
			\sum_{n=3}^\infty
			\frac{q^n }{(1-q^n)^2\aqprod{q^{n+1}}{q}{n}\aqprod{q^{2n+2}}{q^2}{\infty}}		
			\cdot q^{\frac{n(n-3)}{2}}
	,\\
	S_{G3}(q)
	=
		\sum_{n=1}^\infty \spt{G3}{n}q^n	
		&=
		\sum_{n=1}^\infty
		\frac{(-1)^nq^{\frac{n(n+1)}{2}}\aqprod{q^{2n+1}}{q}{\infty}}{\aqprod{q^n,q^n,q^{n+1}}{q}{\infty}}
		\\		
		&=		
		\sum_{n=1}^\infty
		\frac{(-1)^nq^{n}}{(1-q^n)^2\aqprod{q^{n+1}}{q}{\infty}}			
		\cdot\frac{q^{\frac{n(n-1)}{2}}}{\aqprod{q^{n+1}}{q}{\infty}}
		\cdot\frac{1}{\aqprod{q^{n+1}}{q}{n}}
	,\\
	S_{L7}(q)
	=
	\sum_{n=1}^\infty\spt{L7}{n}q^n		
		&=
		\sum_{n=1}^\infty \frac{q^{2n}\aqprod{-q^{2n+1}}{q}{\infty}}{\aqprod{q^{2n}}{q^2}{\infty}^2}
		=
		\sum_{n=1}^\infty \frac{q^{2n}}{(1-q^{2n})^2\aqprod{q^{2n+2}}{q^2}{\infty}}
		\cdot\frac{\aqprod{-q^{2n+1}}{q}{\infty}}{\aqprod{q^{2n+2}}{q^2}{\infty}}
	,\\
	S_{L9}(q)
	=
	\sum_{n=1}^\infty\spt{L9}{n}q^n		
		&=
		\sum_{n=1}^\infty \frac{q^{n}\aqprod{q^{2n+1}}{q^2}{\infty}}{\aqprod{q^{n}}{q}{\infty}^2}
		=
		\sum_{n=1}^\infty \frac{q^{n}}{(1-q^{n})^2\aqprod{q^{n+1}}{q}{\infty}}
		\cdot\frac{1}{\aqprod{q^{n+1}}{q}{n}\aqprod{q^{2n+2}}{q^2}{\infty}}
	,\\
	S_{L12}(q)
	=
	\sum_{n=1}^\infty\spt{L12}{n}q^n		
		&=
		\sum_{n=1}^\infty \frac{q^{2n}\aqprod{-q^{2n+1},-q^{2n+3}}{q^2}{\infty}}{\aqprod{q^{2n}}{q^2}{\infty}^2}
		\\&\quad		
		=
		\sum_{n=1}^\infty \frac{q^{2n}\aqprod{-q^{2n+1}}{q^2}{\infty}}{(1-q^{2n})^2\aqprod{q^{2n+2}}{q^2}{\infty}}
		\cdot\frac{\aqprod{-q^{2n+3}}{q^2}{\infty}}{\aqprod{q^{2n+2}}{q^2}{\infty}}
.
\end{align*}

We now give the combinatorial interpretations of these functions.
We recall that an overpartition is a partition where the first occurence of
a part may be overlined. For example the overpartitions of $3$ are
$3$, $\overline{3}$, $2+1$, $2+\overline{1}$, $\overline{2}+1$, 
$\overline{2}+\overline{1}$, $1+1+1$, and $\overline{1}+1+1$.
We say a vector $(\pi_1,\dots,\pi_k)$, where each $\pi_i$ is a partition
or overpartition, is a vector partition of $n$
if altogether the parts of the $\pi_i$ sum to $n$.
For a partition,
or overpartition, $\pi$ we 
let $s(\pi)$ denote the smallest part of a $\pi$,
$spt(\pi)$ the number of times $s(\pi)$ appears, and 
$\ell(\pi)$ the largest part of $\pi$. 
We use the convention that the empty partition has smallest part $\infty$ and
largest part $0$. Rather than interpret these functions in the order of their definitions, we 
begin with the simplest.

We see $\spt{L7}{n}$ is the number of occurrences of the smallest part in
the partition pairs $(\pi_1,\pi_2)$ of $n$, where $\pi_1$ is a partition into even parts, 
$\pi_2$ is an overpartition
with all non-overlined parts even, and $s(\pi_1)<s(\pi_2)$.
We see $\spt{L9}{n}$ is the number of occurrences of the smallest part in
the partition pairs $(\pi_1,\pi_2)$ of $n$, where $s(\pi_1)<s(\pi_2)$
and all parts of $\pi_2$ larger than $2s(\pi_2)$ must be even.
We see $\spt{L12}{n}$ is the number of occurrences of the smallest part in
the partition pairs $(\pi_1,\pi_2)$ of $n$, where
the odd parts of $\pi_1$ do not repeat, the odd parts of $\pi_2$ do not repeat,
$s(\pi_1)$ is even,  $s(\pi_1)<s(\pi_2)$, and the odd parts of $\pi_2$
are at least $s(\pi_1)+3$.
We see $\spt{G2}{n}$ is the number of occurrences of the smallest part
in the partition pairs $(\pi_1,\pi_2)$ of $n$, where
$s(\pi_1)<s(\pi_2)$,
the parts of $\pi_1$ larger than $2s(\pi_1)$ must be even,
the parts of $\pi_2$ larger than $2s(\pi_1)$ must be $2$ modulo $4$,
and $\pi_2$ has no parts in the interval $(2s(\pi_1),4s(\pi_1)+2)$.

To interpret $\spt{G1}{n}$, we first note that
\begin{align*}
	\frac{q^n}{(1-q^{2n})^2}
	&=
	q^{n}+2q^{3n}+3q^{5n}+4q^{7n}+\dots
	.
\end{align*}
We see $\spt{G1}{n}$ is a weighted count on certain partition triples
$(\pi_1,\pi_2,\pi_3)$ of $n$. Here the restrictions are $spt(\pi_1)$ is odd,
$\pi_1$ has no parts in the interval $(s(\pi_1),2s(\pi_1)+1)$,
$\pi_2$ and $\pi_3$
are partitions with only even parts, $2s(\pi_1)<s(\pi_2)$, $2s(\pi_1)<s(\pi_3)$,
and $\ell(\pi_3)\le 4s(\pi_1)$. 
These partitions tripled are weighted by $\frac{spt(\pi_1)+1}{2}$,
rather than by just $spt(\pi_1)$.

We see $\spt{G3}{n}$ is a weighted count on certain partition triples
$(\pi_1,\pi_2,\pi_3)$ of $n$. Here the restrictions are $\pi_2$ is a partition
where the parts $1,2,\dots,s(\pi_1)-1$ appear exactly once and $s(\pi_1)$
does not appear as a part of $\pi_2$ (with the understanding that this only 
means $s(\pi_2)\ge 2$ when $s(\pi_1)=1$), 
$s(\pi_1)<s(\pi_3)$, and $\ell(\pi_3)\le 2s(\pi_1)$. These partition triples
are weighted by $(-1)^{s(\pi_1)}spt(\pi_1)$. 

We view $\spt{F1}{n}$ as a sum of three functions. The first is 
a weighted count on the partitions $\pi$ of $n$, where $1$ may appear as a 
part but all other parts are even. For this first function these partitions 
are weighted by 
one more than the number of
times the part $1$ appears. The second function is a weighted count on the 
partitions $\pi$ of $n$, where $1$ and $3$ may appear as parts but all other
parts are even and larger than $2$,
and $1$ must appear an odd number of times. For this second function these
partitions are weighted by
$\frac{spt(\pi)+1}{2}$, and we note $spt(\pi)$ is the number of ones.
The third function is a weighted count on the number of partition pairs
$(\pi_1,\pi_2)$ of $n$, where $s(\pi_1)\ge 3$, the parts of
$\pi_1$ larger than $2s(\pi_1)$ must be even, and $\pi_2$ consists of exactly
one copy of $2,3,\dots,s(\pi_1)-2$ (with the understanding that $\pi_2$ is
 the empty partition when $s(\pi_1)=3$). For this third function these
partition pairs are weighted by $spt(\pi_1)$.

\begin{theorem}\label{TheoremCongruences}
For $n\ge0$, we have the following congruences,
\begin{align*}
	\spt{G1}{3n} &\equiv 0\pmod{3}
	,\\	
	\spt{G2}{3n} &\equiv 0\pmod{3}
	,\\
	\spt{G2}{3n+2} &\equiv 0\pmod{3}
	,\\
	\spt{L7}{3n+1} &\equiv 0\pmod{3}
	,\\	
	\spt{L9}{3n+2} &\equiv 0\pmod{3}
	,\\		
	\spt{L12}{3n} &\equiv 0\pmod{3}	
	,\\
	\spt{F1}{10n+9} &\equiv 0\pmod{5}
	,\\
	\spt{G3}{5n+3} &\equiv 0\pmod{5}
	.
\end{align*}
\end{theorem}

To prove the congruences of Theorem \ref{TheoremCongruences} we will prove certain 
identities for the spt-crank two variable series. To explain this,
for $X=L_i,F_1,G_i$
we write
\begin{align*}
	S_{X}(z,q) &=\sum_{n=1}^\infty\sum_{m=-\infty}^\infty M_{X}(m,n)z^mq^n
.
\end{align*}
Whatever each $M_X(m,n)$ is counting is the spt-crank associated to the function
$\spt{X}{n}$.
We define the additional functions
\begin{align*}
	M_{X}(k,t,n) &= \sum_{m\equiv k\pmod{t}}M_{X}(m,n)
.
\end{align*}

For now we consider just $S_{L7}(z,q)$, as the explanations for the other
six functions are identical.
Since $S_{L7}(q)=S_{L7}(1,q)$, we have that
\begin{align*}
	\spt{L7}{n} &= \sum_{k=0}^{t-1} M_{L7}(k,t,n).
\end{align*}
Next with $\zeta_t$ a $t$-th root of unity, we have
\begin{align*}
	S_{L7}(\zeta_t,q)
	&=
	\sum_{n=1}^\infty \left( \sum_{k=0}^{t-1} M_{L7}(k,t,n) \zeta_t^k \right) q^n
	.
\end{align*}
When $t$ is prime and $\zeta_t$ is primitive, the minimal polynomial for 
$\zeta_t$ is $1+x+x^2+\dots+x^{t-1}$. So if the coefficient of $q^N$ in 
$S_{L7}(\zeta_t,q)$ is zero, then
\begin{align*}
	M_{L7}(0,t,N)=M_{L7}(1,t,N)=M_{L7}(2,t,N)=\dots=M_{L7}(t-1,t,N)=\frac{1}{t}\spt{L7}{N}
.
\end{align*}
Since the $M_{L7}(k,t,n)$ are integers, we clearly have $\spt{L7}{N}\equiv 0\pmod{t}$.

That is to say, one way to prove $\spt{L7}{3n+1}\equiv 0\pmod{3}$ is to instead
prove the stronger result that $M_{L7}(0,3,3n+1)=M_{L7}(1,3,3n+1)=M_{L7}(2,3,3n+1)$.
We show these values of the spt-crank are equal by showing the
coefficient of $q^{3n+1}$ in $S_{L7}(\zeta_3,q)$ is zero.
In Section 4 we prove that the coefficients of
$q^{3n+1}$ in $S_{L7}(\zeta_3,q)$,
$q^{3n+2}$ in $S_{L9}(\zeta_3,q)$,
$q^{3n}$ in $S_{L12}(\zeta_3,q)$,
$q^{3n}$ in $S_{G1}(\zeta_3,q)$,
$q^{3n}$ in $S_{G2}(\zeta_3,q)$,
$q^{3n+2}$ in $S_{G2}(\zeta_3,q)$,
$q^{10n+9}$ in $S_{F1}(\zeta_5,q)$, and
$q^{5n+3}$ in $S_{G3}(\zeta_5,q)$
are all zero. This establishes the following Theorem and Theorem
\ref{TheoremCongruences} as a corollary. 
\begin{theorem}\label{TheoremCranks}
For $n\ge 0$, the spt-cranks satisfy the following equalities,
\begin{align*}	
	M_{L7}(0,3,3n+1)&=M_{L7}(1,3,3n+1)=M_{L7}(2,3,3n+1)	
	,\\	
	M_{L9}(0,3,3n+2)&=M_{L9}(1,3,3n+2)=M_{L9}(2,3,3n+2)	
	,\\	
	M_{L12}(0,3,3n)&=M_{L12}(1,3,3n)=M_{L12}(2,3,3n)	
	,\\	
	M_{G1}(0,3,3n)&=M_{G1}(1,3,3n)=M_{G1}(2,3,3n)	
	,\\	
	M_{G2}(0,3,3n)&=M_{G2}(1,3,3n)=M_{G2}(2,3,3n)	
	,\\	
	M_{G2}(0,3,3n+2)&=M_{G2}(1,3,3n+2)=M_{G2}(2,3,3n+2)	
	,\\	
	M_{F1}(0,5,10n+9)&=M_{F1}(1,5,10n+9)=M_{F1}(2,5,10n+9)=M_{F1}(3,5,10n+9)=M_{F1}(4,5,10n+9)	
	,\\	
	M_{G3}(0,5,5n+3)&=M_{G3}(1,5,5n+3)=M_{G3}(2,5,5n+3)=M_{G3}(3,5,5n+3)=M_{G3}(4,5,5n+3)	
.
\end{align*}
\end{theorem}

The main tools to prove Theorem \ref{TheoremCranks} are the following 
identities. We note these are identities for all values of $z$, not just
for $z$ being a specific root of unity.
\begin{theorem}
The spt-crank generating functions can be expressed as the following series,
\label{TheoremMain}
\begin{align}
	\label{EqTheoremSL7}	
	&S_{L7}(z;q)
	\nonumber\\	
		&=
		\frac{1}{(1+z)\aqprod{q^2,z,z^{-1}}{q^2}{\infty}}
		\sum_{j=1}^\infty
		\sum_{n=0}^\infty
		(1-z^j)(1-z^{j-1})z^{1-j} (-1)^{j+1}(1+q^{2j-1})q^{j(j-1) +\frac{n(n-1)}{2}+2jn}
	,\\
	\label{EqTheoremSL9}	
	&S_{L9}(z;q)
		=
		\frac{1}{(1+z)\aqprod{q,z,z^{-1}}{q}{\infty}}
		\sum_{j=1}^\infty
		\sum_{n=0}^\infty
		(1-z^j)(1-z^{j-1})z^{1-j} (-1)^{j+n+1}(1-q^{2j-1})q^{\frac{j(j-1)}{2} + n^2 + 2jn}
	,\\
	\label{EqTheoremSL12}	
	&S_{L12}(z;q)
		=
		\frac{1}{(1+z)\aqprod{q^2,z,z^{-1}}{q^2}{\infty}}
		\sum_{j=1}^\infty
		\sum_{n=0}^\infty
		(1-z^j)(1-z^{j-1})z^{1-j} (-1)^{j+1}(1+q^{2j-1})q^{j(j-1) + n^2 + 2jn}
	,\\
	\label{EqTheoremF}	
	&F(\rho_1,\rho_2,z;q)
		\nonumber\\
		&=
		\sum_{j=1}^\infty
			\frac{(1-z^j)(1-z^{j-1})z^{1-j} (-1)^{j+1}q^{j(j-1)/2}\aqprod{\rho_1,\rho_2}{q}{j-1} 
				\aqprod{ \tfrac{q^{j+1}}{\rho_1}, \tfrac{q^{j+1}}{\rho_2} }{q}\infty	}
			{(1+z) \rho_1^{j-1}\rho_2^{j-1}\aqprod{z,z^{-1}, \rho_1,\rho_2,\tfrac{q}{\rho_1\rho_2}}{q}{\infty}}	
		\nonumber\\&\quad\times
		\left(1 - \tfrac{q^j}{\rho_1} - \tfrac{q^j}{\rho_2} + \tfrac{q^{3j-1}}{\rho_1} + \tfrac{q^{3j-1}}{\rho_2}-q^{4j-2}\right)
	,\\
	\label{EqTheoremG}	
	&G(\rho_1,\rho_2,z;q)
		=
		\sum_{j=1}^\infty
		\frac{(1-z^j)(1-z^{j-1})z^{1-j} (-1)^{j+1}(1-q^{2j-1})q^{\frac{j(j+1)}{2}-1}\aqprod{\rho_1,\rho_2}{q}{j-1} 
			\aqprod{\tfrac{q^{j+1}}{\rho_1},\tfrac{q^{j+1}}{\rho_2}}{q}\infty	}
		{(1+z)  \rho_1^{j-1}\rho_2^{j-1}\aqprod{z,z^{-1}, \rho_1,\rho_2,\tfrac{q^2}{\rho_1\rho_2}}{q}{\infty}}
	,\\
	\label{EqTheoremFInfinity}	
	&F(\rho,z;q)
		=
		\sum_{j=1}^\infty
			\frac{(1-z^j)(1-z^{j-1})z^{1-j} q^{(j-1)^2}\aqprod{\rho}{q}{j-1} 
				\aqprod{ \tfrac{q^{j+1}}{\rho} }{q}\infty	}
			{(1+z)  \rho^{j-1}\aqprod{z,z^{-1}, \rho}{q}{\infty}}	
		\left(1 - \tfrac{q^j}{\rho} + \tfrac{q^{3j-1}}{\rho} -q^{4j-2}\right)
	,\\
	\label{EqTheoremGInfinity}	
	&G(\rho,z;q)
		=
		\sum_{j=1}^\infty
		\frac{(1-z^j)(1-z^{j-1})z^{1-j} (1-q^{2j-1})q^{j(j-1)}\aqprod{\rho}{q}{j-1} 
			\aqprod{\tfrac{q^{j+1}}{\rho}}{q}\infty	}
		{(1+z)  \rho^{j-1} \aqprod{z,z^{-1}, \rho}{q}{\infty}}
	,\\
	\label{EqTheoremS1}
	&S_{G1}(z,q)
		=
		\frac{1}{(1+z)\aqprod{z,z^{-1},q}{q^2}{\infty}}
		\sum_{j=1}^\infty 
		(1-z^j)(1-z^{j-1})z^{1-j}(-1)^{j+1}(1+q^{2j-1})q^{(j-1)^2}
	,\\
	\label{EqTheoremS2}
	&S_{G2}(z,q)	
		=
		\frac{1}{(1+z)\aqprod{z,z^{-1},q}{q}{\infty}}
		\sum_{j=-\infty}^\infty 
		\frac{(1-z^j)(1-z^{j-1})z^{1-j}(-1)^{j+1}q^{\frac{j(j-1)}{2}}}{(1+q^{2j-1})}
	,\\
	\label{EqTheoremS3}
	&S_{F1}(z,q)	
		=
		\frac{1}{(1+z)\aqprod{z,z^{-1}}{q}{\infty}}
		\sum_{j=-\infty}^\infty (1-z^j)(1-z^{j-1})z^{1-j}(-1)^{j+1}q^{j^2-3j+2}(1+q^{j-1})
		\\
		\label{EqTheoremS3Product}		
		&=		
		\frac{\aqprod{z^{-1}q^2,zq^2,q^2}{q^2}{\infty}}{\aqprod{zq,z^{-1}q}{q}{\infty}}		
		+
		\frac{\aqprod{z^{-1}q,zq,q^2}{q^2}{\infty}}{\aqprod{z,z^{-1}}{q}{\infty}}	
		-
		\frac{\aqprod{q,q,q^2}{q^2}{\infty}}{\aqprod{z,z^{-1}}{q}{\infty}}	
	,\\	
	\label{EqTheoremS4}
	&S_{G3}(z,q)	
		=
		\frac{1}{(1+z)\aqprod{z,z^{-1}}{q}{\infty}}
		\sum_{j=1}^\infty (1-z^j)(1-z^{j-1})z^{1-j}(1-q^{2j-1})q^{(j-1)^2}
.	
\end{align}
\end{theorem}
We note (\ref{EqTheoremFInfinity}) and (\ref{EqTheoremGInfinity}) follow 
by taking limits in (\ref{EqTheoremF}) and (\ref{EqTheoremG}) as in the 
definitions of $F(\rho,z;q)$ and $G(\rho,z;q)$.
It is worth pointing out that $S_{F1}(z,q)$ reducing to products tells us that
$S_{F1}(z,q)$ will be a modular form when $z$ is a root of unity. However,
the other spt-cranks with single series representations do not appear to reduce
to products. As such these functions will likely instead be false theta 
functions when $z$ is a root of unity.
The double series identities also have another interesting form.
\begin{theorem}\label{TheoremHeckeSeries}
The following spt-crank functions can be written as Hecke-Rogers double sums,
in particular,
\begin{align}
	\label{EqTheoremSL7Hecke}
	&S_{L7}(z;q)
		=
		\frac{1}{(1+z)\aqprod{q^2,z,z^{-1}}{q^2}{\infty}}
		\sum_{j=0}^\infty
		\sum_{n=-j}^j
		(1-z^{j-|n|+1})(1-z^{j-|n|})z^{|n|-j} (-1)^{j+n}q^{j(j+1) -\frac{n(n-1)}{2}}
	,\\
	\label{EqTheoremSL9Hecke}	
	&S_{L9}(z;q)
		=
		\frac{1}{(1+z)\aqprod{q,z,z^{-1}}{q}{\infty}}
		\sum_{j=0}^\infty
		\sum_{n=-\left\lfloor j/2 \right\rfloor }^{\left\lfloor j/2 \right\rfloor}
		(1-z^{j-2|n|+1})(1-z^{j-2|n|})z^{2|n|-j} (-1)^{j+n}q^{\frac{j(j+1)}{2} -n(n-1) }
	,\\
	\label{EqTheoremSL12Hecke}	
	&S_{L12}(z;q)
		=
		\frac{1}{(1+z)\aqprod{q^2,z,z^{-1}}{q^2}{\infty}}
		\sum_{j=0}^\infty
		\sum_{n=-j}^j
		(1-z^{j-|n|+1})(1-z^{j-|n|})z^{|n|-j} (-1)^{j+n}q^{j(j+1)+n}
.
\end{align}
\end{theorem}
Hecke-Rogers double sums of this form recently arose in \cite{Garvan2} for
the Dyson rank of partitions, the Dyson rank of overpartitions, and the $M_2$-rank
of partitions without repeated odd parts. The identities for those ranks lead
to Hecke-Rogers series for certain related spt functions as well. One simple 
point to notice about the double series in Theorem \ref{TheoremMain} is that
summation indices are independent, but the power of $q$ is a quadratic in 
$j$ and $n$ with a cross term $jn$, whereas the double series in Theorem
\ref{TheoremHeckeSeries} have summation indices that are dependent but the
power of $q$ is a quadratic without a cross term.

Now that we have stated our results, we describe the $q$-series techniques we
need to prove these identities.
We will use Lemma 4.1 of \cite{Garvan2}, which is
\begin{align}
	\label{EqGarvan41}
	\frac{(1+z)\aqprod{z,z^{-1}}{q}{n}}{\aqprod{q}{q}{2n}}
	&=
	\sum_{j=-n}^{n+1}\frac{(-1)^{j+1}(1-q^{2j-1})z^jq^{\frac{j(j-3)}{2}+1} }
	{\aqprod{q}{q}{n+j}\aqprod{q}{q}{n-j+1}}				
.
\end{align}
We recall a pair of sequences $(\alpha,\beta)$ is a Bailey pair relative to
$(a,q)$ if
\begin{align*}
	\beta_n &= \sum_{k=0}^n \frac{\alpha_k}{\aqprod{q}{q}{n-k}\aqprod{aq}{q}{n+k}}
.
\end{align*}
Some authors suppress the dependence of $q$ in the definition of a Bailey pair,
however we will find this additional notation necessary.
We recall a limiting case of Bailey's Lemma states if
$(\alpha,\beta)$ is a Bailey pair relative to $(a,q)$ then
\begin{align}\label{EqBaileysLemma}
	\sum_{n=0}^\infty 
	\aqprod{\rho_1,\rho_2}{q}{n}\left(\tfrac{aq}{\rho_1\rho_2}\right)^n\beta_n
	&=
	\frac{\aqprod{aq/\rho_1,aq/\rho_2}{q}{\infty}}
		{\aqprod{aq,\tfrac{aq}{\rho_1\rho_2}}{q}{\infty}}
	\sum_{n=0}^\infty
	\frac{\aqprod{\rho_1,\rho_2}{q}{n} \left(\tfrac{aq}{\rho_1\rho_2}\right)^n \alpha_n  }
		{\aqprod{aq/\rho_1,aq/\rho_2}{q}{n}}
	.
\end{align}
Bailey pairs and Bailey's Lemma have a rich and varied history. They were 
introduced by Bailey in \cite{Bailey} in reproving the Rogers-Ramanujan 
identities and giving more identities of that type. 
Slater, a student of
Bailey, established what is now a standard list of Bailey pairs in
\cite{Slater1} and \cite{Slater2} and used them to prove a massive list of 
Roger-Ramanujan type identites. Rather than say too much of their history, 
recent developments, and applications, we refer the reader to Chapter 3 of 
\cite{Andrews2} and the survey articles 
\cite{Andrews3,McLaughlinSillsZimmer,Warnaar}.

\begin{lemma}
If $(\alpha,\beta)$ is a Bailey pair relative to $(a,q)$ then
\begin{align}	
	\label{BaileyL9}
	&\sum_{n=0}^\infty
	\aqprod{aq}{q^2}{n}q^n\beta_n
		=
		\frac{1}{\aqprod{aq^2}{q^2}{\infty}\aqprod{q}{q}{\infty}}
		\sum_{r,n\ge 0}(-a)^nq^{n^2+2rn+r+n}\alpha_r
	,\\
	\label{Bailey1}
	&\sum_{n=0}^\infty
		\aqprod{\rho_1\sqrt{a},\rho_2\sqrt{a}}{q}{n}	
		(\tfrac{q}{\rho_1\rho_2})^n
		\beta_n(a,q)
	=
	\frac{\aqprod{\sqrt{a}q/\rho_1,\sqrt{a}q/\rho_2}{q}{\infty}}
		{\aqprod{aq,\tfrac{q}{\rho_1\rho_2}}{q}{\infty}}
	\sum_{n=0}^\infty
	\frac{\aqprod{\rho_1\sqrt{a},\rho_2\sqrt{a}}{q}{n} (\tfrac{q}{\rho_1\rho_2})^n \alpha_n(a,q)}
		{\aqprod{\sqrt{a}q/\rho_1,\sqrt{a}q/\rho_2}{q}{n}}
	,\\
	\label{Bailey2}
	&\sum_{n=0}^\infty
		\aqprod{\rho_1\sqrt{a/q},\rho_2\sqrt{a/q}}{q}{n}	
		(\tfrac{q^2}{\rho_1\rho_2})^n
		\beta_n(a,q)
	\nonumber\\
	&=
	\frac{\aqprod{\sqrt{a}q^{\frac{3}{2}}/\rho_1,\sqrt{a}q^{\frac{3}{2}}/\rho_2}{q}{\infty}}
		{\aqprod{aq,\tfrac{q^2}{\rho_1\rho_2}}{q}{\infty}}
	\sum_{n=0}^\infty
	\frac{\aqprod{\rho_1\sqrt{a/q},\rho_2\sqrt{a/q}}{q}{n} (\tfrac{q^2}{\rho_1\rho_2})^n \alpha_n(a,q)}
		{\aqprod{\sqrt{a}q^{\frac{3}{2}}/\rho_1,\sqrt{a}q^{\frac{3}{2}}/\rho_2}{q}{n}}
.
\end{align}
If $(\alpha,\beta)$ is a Bailey pair relative to $(a^2q^2,q^2)$ then
\begin{align}
	\label{BaileyL7}
	\sum_{n=0}^\infty \aqprod{aq}{q}{n}q^{2n}\beta_n
	&=
	\frac{\aqprod{aq}{q}{\infty}}{\aqprod{a^2q^4}{q^2}{\infty}}
	\sum_{r,n\ge 0}
	\frac{q^{\frac{n(n+1)}{2}+2nr+2r+n}a^n}{1-aq^{2r+1}}\alpha_r	
	.
\end{align}
If $(\alpha,\beta)$ is a Bailey pair relative to $(a^2,q)$ then
\begin{align}
	\label{BaileyL12}
	\sum_{n=0}^\infty 
	\frac{\aqprod{a^2}{q}{2n}q^{n}}{\aqprod{a,aq}{q}{n}}\beta_n
	&=
	\frac{1}{\aqprod{q,aq,aq}{q}{\infty}}
	\sum_{r,n\ge 0}
	\frac{(-a)^nq^{\frac{n(n+1)}{2}+nr+r}(1+a)}{1+aq^r}\alpha_r	
	.
\end{align}
\end{lemma}
\begin{proof}
Equations (\ref{BaileyL9}), (\ref{BaileyL7}), and (\ref{BaileyL12})
are exactly (1.9), (1.7),  and (1.12) of \cite{Lovejoy}.
We find (\ref{Bailey1}) follows from (\ref{EqBaileysLemma}) by letting
$\rho_1\mapsto\rho_1\sqrt{a}$ and $\rho_2\mapsto\rho_2\sqrt{a}$ and
(\ref{Bailey2}) follows from (\ref{EqBaileysLemma}) by letting
$\rho_1\mapsto\rho_1\sqrt{a/q}$ and $\rho_2\mapsto\rho_2\sqrt{a/q}$.
\end{proof}

We only need the following two Bailey pairs relative to $(a,q)$,
\begin{align}
\label{FirstBaileyPair}
	\beta^*_n(a,q) &= \frac{1}{\aqprod{aq,q}{q}{n}}
	,
	&\alpha^*_n(a,q) &= \PieceTwo{1}{0}{n=0,}{n\ge 1,}
	\\
	\label{SecondBaileyPair}
	\beta^{**}_n(a,q) &= \frac{1}{\aqprod{aq^2,q}{q}{n}}
	,
	&\alpha^{**}_n(a,q) &= 
		\left\{
   		\begin{array}{ll}
      		1 & n=0 
      		,\\
       		-aq & n=1 
       		,\\
			0 & n\ge 2. 
     		\end{array}
		\right.
\end{align}
That these are Bailey pairs relative to $(a,q)$ follows immediately from the
definition of a Bailey pair. 
We next explain how our spt-cranks were found.

\section{The general idea}

Bailey's Lemma comes from Bailey's Transform, which states that if
\begin{align}\label{OriginalBaileyPairs}
	\beta_n 
	&= 
		\sum_{k=0}^n \alpha_k u_{n-k} v_{n+k}
	,&
	\gamma_n 
	&= 
		\sum_{k=n}^\infty \delta_k u_{k-n} v_{k+n}
	,
\end{align}
then
\begin{align*}
	\sum_{n=0}^\infty \beta_n\delta_n
	&=
	\sum_{n=0}^\infty \alpha_n\gamma_n
.
\end{align*}
If in (\ref{OriginalBaileyPairs}) we use $u_n=1/\aqprod{q}{q}{n}$ and
$v_n=1/\aqprod{aq}{q}{n}$, then we see the condition on $\alpha$ and $\beta$ is 
exactly that of being a Bailey pair. We refer to $(\gamma,\delta)$ as a 
conjugate Bailey pair. While the idea of a conjugate Bailey pair is clearly
built into the theory of Bailey pairs, they were almost completely ignored 
until Schilling and Warnaar brought attention to them in 
\cite{SchillingWarnaar}. A simple statement about conjugate Bailey pairs is
that for each conjugate Bailey pair, we have a new version of Bailey's Lemma.

Our method relies on applying these ``new'' versions of Bailey's Lemma. 
This was born out of the proofs of \cite{Garvan2} and 
\cite{GarvanJennings2}. In \cite{GarvanJennings2} we defined an spt-crank-type 
function to be a series of the form
\begin{align*}
	\frac{P(q)}{\aqprod{z,z^{-1}}{q}{\infty}}
	\sum_{n=1}^\infty \aqprod{z,z^{-1}}{q}{n}q^n\beta_n
,
\end{align*}  
where $P(q)$ is some infinite product and $\beta$ comes from a Bailey pair
relative to $(1,q)$. In that article and others, we chose a Bailey pair of 
Slater \cite{Slater1,Slater2},
applied Bailey's Lemma to the spt-crank-type function, and obtained a generalized
Lambert series and an infinite product to work with. Instead, here we take an spt-crank
function to be of the form
\begin{align*}
	S_A(z,q)
	&=
	\frac{P(q)}{\aqprod{z,z^{-1}}{q}{\infty}}
	\sum_{n=1}^\infty \aqprod{z,z^{-1}}{q}{n}q^n A_n(q)
,
\end{align*}  
and do not worry about $A$ occurring as part of a Bailey pair, we only require 
that it be a function of $q$ and $n$. We instead ask how
should we choose $A$ so that we can transform the spt-crank function 
with a conjugate Bailey pair identity using one of the Bailey pairs
$(\alpha^*,\beta^*)$ or $(\alpha^{**},\beta^{**})$. This method leads to the
series identities in Theorem \ref{TheoremMain}.

Suppose one takes a fixed conjugate Bailey pair and the 
resulting Bailey's Lemma type identity. To be explicit we take the identity 
(1.9) from \cite{Lovejoy},
\begin{align}\label{EqLovejoy19Again}
	&\sum_{n=0}^\infty
	\aqprod{aq}{q^2}{n}q^n\beta_n
		=
		\frac{1}{\aqprod{aq^2}{q^2}{\infty}\aqprod{q}{q}{\infty}}
		\sum_{r,n\ge 0}(-a)^nq^{n^2+2rn+r+n}\alpha_r
	,
\end{align}
where $(\alpha,\beta)$ is a Bailey pair relative to $(a,q)$.
First we apply (\ref{EqGarvan41}) to $S_A(z,q)$ to find that
\begin{align*}
	(1+z)\aqprod{z,z^{-1}}{q}{\infty}S_{A}(z,q)
	&=
		P(q)\sum_{n=1}^\infty A_n(q) q^n \aqprod{q}{q}{2n}
		\sum_{j=-n}^{n+1}
		\frac{ (-1)^{j+1} (1-q^{2j-1}) z^j q^{\frac{j(j-3)}{2}+1} }
			{\aqprod{q}{q}{n+j}\aqprod{q}{q}{n-j+1}}
.
\end{align*}
We find that the coefficients of $z^j$ and $z^{1-j}$ in 
$(1+z)\aqprod{z,z^{-1}}{q}{\infty}S_{A}(z,q)$ are equal, due to
$S_A(z,q)$ being symmetric in $z$ and $z^{-1}$.
For $j\ge 2$ we find the coefficient of $z^j$ in 
$(1+z)\aqprod{z,z^{-1}}{q}{\infty}S_A(z,q)$ is given by
\begin{align*}
	&P(q)(-1)^{j+1}(1-q^{2j-1})q^{\frac{j(j-3)}{2}+1}
	\sum_{n=j-1}^\infty
	\frac{A_n(q) q^n \aqprod{q}{q}{2n}}{\aqprod{q}{q}{n+j}\aqprod{q}{q}{n-j+1}}
	\\
	&=
		P(q)(-1)^{j+1}(1-q^{2j-1})q^{\frac{j(j-3)}{2}+1}	
		\sum_{n=0}^\infty
		\frac{A_{n+j-1}(q) q^{n+j-1}\aqprod{q}{q}{2n+2j-2}}
			{\aqprod{q}{q}{n+2j-1}\aqprod{q}{q}{n}}
	\\
	&=
		P(q)(-1)^{j+1} q^{\frac{j(j-1)}{2}}
		\sum_{n=0}^\infty
		\frac{A_{n+j-1}(q) q^n\aqprod{q^{2j-1}}{q}{2n}}
			{\aqprod{q^{2j},q}{q}{n}}
	\\
	&=
		P(q)(-1)^{j+1} q^{\frac{j(j-1)}{2}}
		\sum_{n=0}^\infty
		A_{n+j-1}(q) q^n \aqprod{q^{2j-1}}{q}{2n} \beta^*_n(q^{2j-1};q)
	\\
	&=
		P(q)(-1)^{j+1} q^{\frac{j(j-1)}{2}}
		\sum_{n=0}^\infty
		A_{n+j-1}(q) q^n \aqprod{q^{2j-1}}{q}{2n} \beta^{**}_n(q^{2j-2};q)
.
\end{align*}
The coefficient of $z$ is derived in the same way, but because $n$ will start
at $1$, rather than $0$, we find the coefficient of $z$ is given by 
\begin{align*}
	P(q)
	\sum_{n=0}^\infty
		A_{n}(q) q^n \aqprod{q}{q}{2n} \beta^*_n(q;q)
		-
		A_0(q) P(q)
	&=
	P(q)
	\sum_{n=0}^\infty
		A_{n}(q) q^n \aqprod{q}{q}{2n} \beta^{**}_n(1;q)
		-
		A_0(q)P(q)
.
\end{align*}

If we are to next apply (\ref{EqLovejoy19Again}) with $(\alpha^*,\beta^*)$, 
then we must choose $A$ so that
\begin{align*}
	A_{n+j-1}(q) q^n \aqprod{q^{2j-1}}{q}{2n} 
	&=	
		C_j(q)\aqprod{q^{2j}}{q^2}{n}q^n	
,
\end{align*}
where $C$ is a function dependent on $j$, but not dependent on $n$.
Since
\begin{align*}
	\frac{\aqprod{q^{2j}}{q^2}{n}}
		{\aqprod{q^{2j-1}}{q}{2n}	}
	&=
	\frac{\aqprod{q^2}{q^2}{n+j-1} }
		{\aqprod{q}{q}{2(n+j-1)}	}
	\frac{\aqprod{q}{q}{2j-2}}{\aqprod{q^2}{q^2}{j-1}}	
	,
\end{align*}
we see a reasonable choice for $A$ is
\begin{align*}
	A_n(q) 
	&= 
		\frac{\aqprod{q^2}{q^2}{n}}{\aqprod{q}{q}{2n}}
	=
		\frac{1}{\aqprod{q}{q^2}{n}}
.
\end{align*}
If instead we want apply (\ref{EqLovejoy19Again}) with $(\alpha^{**},\beta^{**})$, 
then we must choose $A$ so that
\begin{align*}
	A_{n+j-1}(q) q^n \aqprod{q^{2j-1}}{q}{2n} 
	&=	
		C_j(q)\aqprod{q^{2j-1}}{q^2}{n}	q^n
.
\end{align*}
Here we find a reasonable choice for $A$ is 
\begin{align*}
	A_n(q) &= \frac{\aqprod{q}{q^2}{n}}{\aqprod{q}{q}{2n}}
	=
	\frac{1}{\aqprod{q^2}{q^2}{n}}
.
\end{align*}

Which of these two choices should we use? Both are valid and lead to different
functions. Using $(\alpha^*,\beta^*)$ we would likely consider the function
\begin{align*}
	\frac{\aqprod{q}{q^2}{\infty}}{\aqprod{z,z^{-1}}{q}{\infty}}
	\sum_{n=1}^\infty \frac{\aqprod{z,z^{-1}}{q}{n}q^n}{\aqprod{q}{q^2}{n}}
\end{align*}
and using $(\alpha^{**},\beta^{**})$ we would use
\begin{align*}
	\frac{\aqprod{q^2}{q^2}{\infty}}{\aqprod{z,z^{-1}}{q}{\infty}}
	\sum_{n=1}^\infty \frac{\aqprod{z,z^{-1}}{q}{n}q^n}{\aqprod{q^2}{q^2}{n}}
.
\end{align*}
However, after elementary rearrangements we recognize the latter as the 
overpartition spt-crank function of \cite{GarvanJennings} and we have nothing
new. With the former, which we earlier called $S_{L9}(z,q)$,
we reduce the products, set $z=1$, and see if we observe
any congruences on a computer. Given the empirical evidence of the congruence
$\spt{L9}{3n+2}$, we then check if the coefficients of $q^{3n+2}$ of 
$S_{L9}(\zeta_3,q)$ appear to be zero. These coefficients do appear to be zero and
so we have a new candidate spt function and spt-crank function to work with.
We now apply (\ref{EqLovejoy19Again}) with $(\alpha^*,\beta^*)$, sum the powers
of $z$, and after a bit of work we arrive at the identity for $S_{L9}(z,q)$ in 
Theorem \ref{TheoremMain}.

We repeat this process with the other conjugate Bailey pair identities of 
\cite{Lovejoy}, as well as with the classical form of Bailey's Lemma 
(\ref{EqBaileysLemma}). After
determining choices of $A_n(q)$ for $(\alpha^*,\beta^*)$ and 
$(\alpha^{**},\beta^{**})$, we see what functions they reduce to when $z=1$,
and check for congruences. Doing so gives many functions previously studied and
many functions without apparent congruences. The functions without apparent 
congruences will still satisfy identities like those in Theorem 
\ref{TheoremMain}, however there are far too many to go through unless we have
a vested interest.

\section{Proof of Theorem \ref{TheoremMain} and \ref{TheoremHeckeSeries}}

\begin{proof}[Proof of (\ref{EqTheoremSL7})]
We use the rearrangements described in the previous section with
(\ref{BaileyL7}) applied to $\beta^{*}(q^{2j-2},q^2)$.
We find the coefficient of $z^j$, for $j\ge 2$, of 
$(1+z)\aqprod{z,z^{-1}}{q^2}{\infty}S_{L7}(z;q)$
is given by
\begin{align*}
	\aqprod{-q}{q}{\infty}
	\sum_{n=j-1}^\infty
	\frac{ (-1)^{j+1}(1-q^{4j-2})q^{2n+j(j-3)+2}\aqprod{q^2}{q^2}{2n}}
		{\aqprod{-q}{q}{2n}\aqprod{q^2}{q^2}{n+j}\aqprod{q^2}{q^2}{n-j+1}}	
	&=
	\frac{(-1)^{j+1}(1+q^{2j-1})q^{j(j-1)}}
	{\aqprod{q^2}{q^2}{\infty}}
	\sum_{n=0}^\infty	q^{\frac{n(n-1)}{2} +2jn }
	.
\end{align*}
When $j=1$ we must subtract $\aqprod{-q}{q}{\infty}$ from this.
By summing the powers of $z$ we then find that
\begin{align*}
	&(1+z)\aqprod{z,z^{-1}}{q^2}{\infty}S_{L7}(z;q)
	\\
	&=
	-(1+z)\aqprod{-q}{q}{\infty}
	+
	\frac{1}{\aqprod{q^2}{q^2}{\infty}}
	\sum_{j=1}^\infty
	(z^j+z^{1-j}) (-1)^{j+1}(1+q^{2j-1})q^{j(j-1)}
	\sum_{n=0}^\infty q^{\frac{n(n-1)}{2} +2jn }	
.
\end{align*}
However, we note the left hand side is zero when $z=1$, and so
\begin{align*}
	\aqprod{-q}{q}{\infty}
	&=	
	\frac{1}{\aqprod{q^2}{q^2}{\infty}}
	\sum_{j=1}^\infty
	(-1)^{j+1}(1+q^{2j-1})q^{j(j-1)}
	\sum_{n=0}^\infty q^{\frac{n(n-1)}{2} +2jn }	
	.
\end{align*}
Noting $z^j+z^{1-j}-1-z=(1-z^j)(1-z^{j-1})z^{1-j}$, we then have
\begin{align*}
	&(1+z)\aqprod{z,z^{-1}}{q^2}{\infty}S_{L7}(z;q)
	\\
	&=
	\frac{1}{\aqprod{q^2}{q^2}{\infty}}
	\sum_{j=1}^\infty
	(1-z^j)(1-z^{j-1})z^{1-j} (-1)^{j+1}(1+q^{2j-1})q^{j(j-1)}
	\sum_{n=0}^\infty q^{\frac{n(n-1)}{2} +2jn }	
	,
\end{align*}
which immediately implies (\ref{EqTheoremSL7}).
\end{proof}
\begin{proof}[Proof of (\ref{EqTheoremSL9})]
We use the rearrangements described in the previous section with
(\ref{BaileyL9}) applied to $\beta^{*}(q^{2j-1},q)$.
We find the coefficient of $z^j$, for $j\ge 2$, of
$(1+z)\aqprod{z,z^{-1}}{q}{\infty}S_{L9}(z;q)$
is given by
\begin{align*}
	\aqprod{q}{q^2}{\infty}
	\sum_{n=j-1}^\infty
	\frac{ (-1)^{j+1}(1-q^{2j-1})q^{n+j(j-3)/2+1}\aqprod{q}{q}{2n}}
		{\aqprod{q}{q^2}{n}\aqprod{q}{q}{n+j}\aqprod{q}{q}{n-j+1}}	
	&=
	\frac{(-1)^{j+1}(1-q^{2j-1})q^{j(j-1)/2}}
	{\aqprod{q}{q}{\infty}}
	\sum_{n=0}^\infty
		(-1)^n	q^{n^2+2jn }
	.
\end{align*}
When $j=1$ we must subtract $\aqprod{q}{q^2}{\infty}$ from this.
By summing the powers of $z$ we then find that
\begin{align*}
	&(1+z)\aqprod{z,z^{-1}}{q}{\infty}S_{L9}(z;q)
	\\
	&=
	-(1+z)\aqprod{q}{q^2}{\infty}
	+
	\frac{1}{\aqprod{q}{q}{\infty}}
	\sum_{j=1}^\infty
	(z^j+z^{1-j})(-1)^{j+1}(1-q^{2j-1})q^{j(j-1)/2}
	\sum_{n=0}^\infty
		(-1)^n	q^{n^2+2jn }
.
\end{align*}
However, we note the left hand side is zero when $z=1$, and so
\begin{align*}
	\aqprod{q}{q^2}{\infty}
	&=	
	\frac{1}{\aqprod{q}{q}{\infty}}
	\sum_{j=1}^\infty
	(-1)^{j+1}(1-q^{2j-1})q^{j(j-1)/2}
	\sum_{n=0}^\infty
		(-1)^n	q^{n^2+2jn }
	.
\end{align*}
We then have
\begin{align*}
	&(1+z)\aqprod{z,z^{-1}}{q}{\infty}S_{L9}(z;q)
	\\
	&=
	\frac{1}{\aqprod{q}{q}{\infty}}
	\sum_{j=1}^\infty
	(1-z^j)(1-z^{j-1})z^{1-j}
	(-1)^{j+1}(1-q^{2j-1})q^{j(j-1)/2}
	\sum_{n=0}^\infty
		(-1)^n	q^{n^2+2jn }	
	,
\end{align*}
which immediately implies (\ref{EqTheoremSL9}).
\end{proof}
\begin{proof}[Proof of (\ref{EqTheoremSL12})]
We use the rearrangements described in the previous section with
(\ref{BaileyL12}) applied to $\beta^{*}(q^{4j-2},q^2)$.
We find the coefficient of $z^j$, for $j\ge 2$, of
$(1+z)\aqprod{z,z^{-1}}{q}{\infty}S_{L12}(z;q)$
is given by
\begin{align*}
	\aqprod{q}{q^2}{\infty}^2
	\sum_{n=j-1}^\infty
	\frac{ (-1)^{j+1}(1-q^{4j-2})q^{2n+j(j-3)+2}\aqprod{q^2}{q^2}{2n}}
		{\aqprod{q}{q^2}{n}\aqprod{q}{q^2}{n+1}\aqprod{q^2}{q^2}{n+j}\aqprod{q^2}{q^2}{n-j+1}}	
	&=
	\frac{(-1)^{j+1}(1-q^{2j-1})q^{j(j-1)}}
	{\aqprod{q^2}{q^2}{\infty}}
	\sum_{n=0}^\infty
		(-1)^n q^{n^2+2jn }  
	.
\end{align*}
When $j=1$ we must subtract $\aqprod{q,q^3}{q^2}{\infty}$ from this.
By summing the powers of $z$ we then find that
\begin{align*}
	&(1+z)\aqprod{z,z^{-1}}{q}{\infty}S_{L12}(z;q)
	\\
	&=
	-(1+z)\aqprod{q,q^3}{q^2}{\infty}
	+
	\frac{1}{\aqprod{q^2}{q^2}{\infty}}
	\sum_{j=1}^\infty
	(z^j+z^{1-j})(-1)^{j+1}(1-q^{2j-1})q^{j(j-1)}
	\sum_{n=0}^\infty
		(-1)^n	q^{n^2+2jn }
.
\end{align*}
However, we note the left hand side is zero when $z=1$, and so
\begin{align*}
	\aqprod{q,q^3}{q^2}{\infty}
	&=	
	\frac{1}{\aqprod{q^2}{q^2}{\infty}}
	\sum_{j=1}^\infty
	(-1)^{j+1}(1-q^{2j-1})q^{j(j-1)}
	\sum_{n=0}^\infty
		(-1)^n	q^{n^2+2jn }
	.
\end{align*}
We then have
\begin{align*}
	&(1+z)\aqprod{z,z^{-1}}{q^2}{\infty}S_{L12}(z;q)
	\\
	&=
	\frac{1}{\aqprod{q^2}{q^2}{\infty}}
	\sum_{j=1}^\infty
	(1-z^j)(1-z^{j-1})z^{1-j}
	(-1)^{j+1}(1-q^{2j-1})q^{j(j-1)}
	\sum_{n=0}^\infty
		(-1)^n	q^{n^2+2jn }	
	,
\end{align*}
which immediately implies (\ref{EqTheoremSL12}).
\end{proof}
\begin{proof}[Proof of (\ref{EqTheoremF})]
We use the rearrangements described in the previous section with
(\ref{Bailey1}) applied to $\beta^{**}(q^{2j-2},q)$.
We find the coefficient of $z^j$, for $j\ge 2$, of
$(1+z)\aqprod{z,z^{-1}}{q}{\infty}F(\rho_1,\rho_2,z;q)$
is given by
\begin{align*}
	&\frac{\aqprod{q}{q}{\infty}}{\aqprod{\rho_1,\rho_2}{q}{\infty}}
	\sum_{n=j-1}^\infty
	\frac{ \aqprod{\rho_1,\rho_2}{q}{n} (\tfrac{q}{\rho_1\rho_2})^n
		(-1)^{j+1}(1-q^{2j-1})q^{\frac{j(j-3)}{2}+1}	
	}{\aqprod{q}{q}{n+j}\aqprod{q}{q}{n-j+1}}
	\\
	&=
	\frac{(-1)^{j+1}q^{j(j-1)/2}\aqprod{\rho_1,\rho_2}{q}{j-1} \aqprod{q^{j+1}/\rho_1,q^{j+1}/\rho_2}{q}\infty	}
	{\rho_1^{j-1}\rho_2^{j-1}\aqprod{\rho_1,\rho_2,\tfrac{q}{\rho_1\rho_2}}{q}{\infty}}	
	(1-q^j/\rho_1-q^j/\rho_2+q^{3j-1}/\rho_1+q^{3j-1}/\rho_2-q^{4j-2})
	.
\end{align*}
When $j=1$ we must subtract $\aqprod{q}{q}{\infty}/\aqprod{\rho_1,\rho_2}{q}{\infty}$ from this.
By summing the powers of $z$ we then find that
\begin{align*}
	&(1+z)\aqprod{z,z^{-1}}{q}{\infty}F(\rho_1,\rho_2,z;q)
	\\
	&=
	-(1+z)\frac{\aqprod{q}{q}{\infty}}{\aqprod{\rho_1,\rho_2}{q}{\infty}}
	\\&\quad
	+
	\sum_{j=1}^\infty
		(z^j+z^{1-j})
		\frac{(-1)^{j+1}q^{j(j-1)/2}\aqprod{\rho_1,\rho_2}{q}{j-1} 
			\aqprod{ \tfrac{q^{j+1}}{\rho_1}, \tfrac{q^{j+1}}{\rho_2} }{q}\infty	}
		{\rho_1^{j-1}\rho_2^{j-1}\aqprod{\rho_1,\rho_2,\tfrac{q}{\rho_1\rho_2}}{q}{\infty}}	
	\left(1 - \tfrac{q^j}{\rho_1} - \tfrac{q^j}{\rho_2} + \tfrac{q^{3j-1}}{\rho_1} + \tfrac{q^{3j-1}}{\rho_2}-q^{4j-2}\right)
.
\end{align*}
However, we note the left hand side is zero when $z=1$, and so
\begin{align*}
	\frac{\aqprod{q}{q}{\infty}}{\aqprod{\rho_1,\rho_2}{q}{\infty}}
	&=	
	\sum_{j=1}^\infty
		\frac{(-1)^{j+1}q^{j(j-1)/2}\aqprod{\rho_1,\rho_2}{q}{j-1} 
			\aqprod{ \tfrac{q^{j+1}}{\rho_1}, \tfrac{q^{j+1}}{\rho_2} }{q}\infty	}
		{\rho_1^{j-1}\rho_2^{j-1}\aqprod{\rho_1,\rho_2,\tfrac{q}{\rho_1\rho_2}}{q}{\infty}}	
	\left(1 - \tfrac{q^j}{\rho_1} - \tfrac{q^j}{\rho_2} + \tfrac{q^{3j-1}}{\rho_1} + \tfrac{q^{3j-1}}{\rho_2}-q^{4j-2}\right)
	.
\end{align*}
We then have
\begin{align*}
	&(1+z)\aqprod{z,z^{-1}}{q}{\infty}F(\rho_1,\rho_2,z;q)
	\\
	&=
	\sum_{j=1}^\infty
		\frac{(1-z^j)(1-z^{j-1})z^{1-j} (-1)^{j+1}q^{j(j-1)/2}\aqprod{\rho_1,\rho_2}{q}{j-1} 
			\aqprod{ \tfrac{q^{j+1}}{\rho_1}, \tfrac{q^{j+1}}{\rho_2} }{q}\infty	}
		{\rho_1^{j-1}\rho_2^{j-1}\aqprod{\rho_1,\rho_2,\tfrac{q}{\rho_1\rho_2}}{q}{\infty}}	
	\\&\quad\times
	\left(1 - \tfrac{q^j}{\rho_1} - \tfrac{q^j}{\rho_2} + \tfrac{q^{3j-1}}{\rho_1} + \tfrac{q^{3j-1}}{\rho_2}-q^{4j-2}\right)
	,
\end{align*}
which immediately implies (\ref{EqTheoremF}).
\end{proof}
\begin{proof}[Proof of (\ref{EqTheoremG})]
We use the rearrangements described in the previous section with
(\ref{Bailey2}) applied to $\beta^{*}(q^{2j-1},q)$.
We find the coefficient of $z^j$, for $j\ge 2$, of
$(1+z)\aqprod{z,z^{-1}}{q}{\infty}G(\rho_1,\rho_2,z;q)$
is given by
\begin{align*}
	&\frac{\aqprod{q}{q}{\infty}}{\aqprod{\rho_1,\rho_2}{q}{\infty}}
	\sum_{n=j-1}^\infty
	\frac{ \aqprod{\rho_1,\rho_2}{q}{n} (\tfrac{q^2}{\rho_1\rho_2})^n
		(-1)^{j+1}(1-q^{2j-1})q^{\frac{j(j-3)}{2}+1}	
	}{\aqprod{q}{q}{n+j}\aqprod{q}{q}{n-j+1}}
	\\
	&=
	\frac{(-1)^{j+1}(1-q^{2j-1})q^{j(j+1)/2-1}\aqprod{\rho_1,\rho_2}{q}{j-1} \aqprod{q^{j+1}/\rho_1,q^{j+1}/\rho_2}{q}\infty	}
	{\rho_1^{j-1}\rho_2^{j-1}\aqprod{\rho_1,\rho_2,\tfrac{q^2}{\rho_1\rho_2}}{q}{\infty}}	
	.
\end{align*}
When $j=1$ we must subtract $\aqprod{q}{q}{\infty}/\aqprod{\rho_1,\rho_2}{q}{\infty}$ from this.
By summing the powers of $z$ we then find that
\begin{align*}
	&(1+z)\aqprod{z,z^{-1}}{q}{\infty}G(\rho_1,\rho_2,z;q)
	\\
	&=
	-(1+z)\frac{\aqprod{q}{q}{\infty}}{\aqprod{\rho_1,\rho_2}{q}{\infty}}
	+
	\sum_{j=1}^\infty
	(z^j+z^{1-j})
	\frac{(-1)^{j+1}(1-q^{2j-1})q^{j(j+1)/2-1}\aqprod{\rho_1,\rho_2}{q}{j-1} 
		\aqprod{q^{j+1}/\rho_1,q^{j+1}/\rho_2}{q}\infty	}
	{\rho_1^{j-1}\rho_2^{j-1}\aqprod{\rho_1,\rho_2,\tfrac{q^2}{\rho_1\rho_2}}{q}{\infty}}	
.
\end{align*}
However, we note the left hand side is zero when $z=1$, and so
\begin{align*}
	\frac{\aqprod{q}{q}{\infty}}{\aqprod{\rho_1,\rho_2}{q}{\infty}}
	&=	
	\sum_{j=1}^\infty
	\frac{(-1)^{j+1}(1-q^{2j-1})q^{j(j+1)/2-1}\aqprod{\rho_1,\rho_2}{q}{j-1} 
		\aqprod{q^{j+1}/\rho_1,q^{j+1}/\rho_2}{q}\infty	}
	{\rho_1^{j-1}\rho_2^{j-1}\aqprod{\rho_1,\rho_2,\tfrac{q^2}{\rho_1\rho_2}}{q}{\infty}}	
	.
\end{align*}
We note this is the same product as in the previous proof, but a rather different 
series. Regardless, we then have
\begin{align*}
	&(1+z)\aqprod{z,z^{-1}}{q}{\infty}G(\rho_1,\rho_2,z;q)
	\\
	&=
	\sum_{j=1}^\infty
	\frac{(1-z^j)(1-z^{j-1})z^{1-j} (-1)^{j+1}(1-q^{2j-1})q^{j(j+1)/2-1}\aqprod{\rho_1,\rho_2}{q}{j-1} 
		\aqprod{q^{j+1}/\rho_1,q^{j+1}/\rho_2}{q}\infty	}
	{\rho_1^{j-1}\rho_2^{j-1}\aqprod{\rho_1,\rho_2,\tfrac{q^2}{\rho_1\rho_2}}{q}{\infty}}	
	,
\end{align*}
which immediately implies (\ref{EqTheoremG}).
\end{proof}
\begin{proof}[Proof of (\ref{EqTheoremS1}).]
With $q\mapsto q^2$, $\rho_1=q$, and $\rho_2=q^2$ in (\ref{EqTheoremG}), we 
have that
\begin{align*}
	&S_{G1}(z,q)
	\\
	&=
	\frac{1}{(1+z)\aqprod{z,z^{-1}}{q^2}{\infty}}
	\sum_{j=1}^\infty
	\frac{(1-z^j)(1-z^{j-1})z^{1-j}(-1)^{j+1}(1-q^{4j-2})q^{j(j+1)-2}\aqprod{q,q^2}{q}{j-1}
		\aqprod{q^{2j+1},q^{2j}}{q^2}{\infty} }
	{q^{3j-3}\aqprod{q,q^2,q}{q^2}{\infty}}
	\\
	&=
	\frac{1}{(1+z)\aqprod{z,z^{-1},q}{q^2}{\infty}}
	\sum_{j=1}^\infty
	(1-z^j)(1-z^{j-1})z^{1-j}(-1)^{j+1}(1+q^{2j-1})q^{(j-1)^2}	
,
\end{align*}
which is (\ref{EqTheoremS1}).
\end{proof}
\begin{proof}[Proof of (\ref{EqTheoremS2}).]
With $\rho_1=iq^{1/2}$, and $\rho_2=-iq^{1/2}$ in (\ref{EqTheoremG}), we 
have that
\begin{align*}
	&S_{G2}(z,q)
	\\
	&=
	\frac{1}{(1+z)\aqprod{z,z^{-1}}{q}{\infty}}
	\sum_{j=1}^\infty
		(1-z^j)(1-z^{j-1})z^{1-j}(-1)^{j+1}(1-q^{2j-1})q^{\frac{j(j+1)}{2}-1}
		\aqprod{iq^{1/2},-iq^{1/2}}{q}{j-1}
		\\&\quad\times
		\frac{ \aqprod{-iq^{j+1/2},iq^{j+1/2}}{q}{\infty} }
		{q^{j-1}\aqprod{iq^{1/2},-iq^{1/2},q}{q}{\infty}}
	\\
	&=
	\frac{1}{(1+z)\aqprod{z,z^{-1}}{q}{\infty}}
	\sum_{j=1}^\infty
	\frac{(1-z^j)(1-z^{j-1})z^{1-j}(-1)^{j+1}(1-q^{2j-1})q^{\frac{j(j-1)}{2}}
		\aqprod{-q}{q^2}{j-1}
		\aqprod{-q^{2j+1}}{q^2}{\infty} }
	{\aqprod{-q}{q^2}{\infty}\aqprod{q}{q}{\infty}}
	\\
	&=
	\frac{1}{(1+z)\aqprod{z,z^{-1},q}{q}{\infty}}
	\sum_{j=1}^\infty
	\frac{(1-z^j)(1-z^{j-1})z^{1-j}(-1)^{j+1}(1-q^{2j-1})q^{\frac{j(j-1)}{2}}}
	{(1+q^{2j-1})}
	\\
	&=
	\frac{1}{(1+z)\aqprod{z,z^{-1},q}{q}{\infty}}
	\sum_{j=-\infty}^\infty
	\frac{(1-z^j)(1-z^{j-1})z^{1-j}(-1)^{j+1}q^{\frac{j(j-1)}{2}}}
	{(1+q^{2j-1})}
,
\end{align*}
which is (\ref{EqTheoremS2}).
\end{proof}
\begin{proof}[Proof of (\ref{EqTheoremS3}).]
With $\rho=-q$ in (\ref{EqTheoremFInfinity}), we 
have that
\begin{align*}
	S_{F1}(z,q)
	&=
	\frac{1}{(1+z)\aqprod{z,z^{-1}}{q}{\infty}}
	\sum_{j=1}^\infty
	\frac{(1-z^j)(1-z^{j-1})z^{1-j}(-1)^{j+1}q^{(j-1)^2}\aqprod{-q}{q}{j-1}
		\aqprod{-q^j}{q}{\infty} }
	{q^{j-1}\aqprod{-q}{q}{\infty}}
		\\&\quad\times
		(1+q^{j-1}-q^{3j-2}-q^{4j-2})
	\\
	&=
	\frac{1}{(1+z)\aqprod{z,z^{-1}}{q}{\infty}}
	\sum_{j=1}^\infty
	(1-z^j)(1-z^{j-1})z^{1-j}(-1)^{j+1}(1+q^{j-1}-q^{3j-2}-q^{4j-2})q^{j^2-3j+2}	
	\\
	&=
	\frac{1}{(1+z)\aqprod{z,z^{-1}}{q}{\infty}}
	\sum_{j=-\infty}^\infty
	(1-z^j)(1-z^{j-1})z^{1-j}(-1)^{j+1}(1+q^{j-1})q^{j^2-3j+2}	
,
\end{align*}
which is (\ref{EqTheoremS3}). Next we will use the Jacobi triple product identity
\cite[Theorem 2.8]{AndrewsBook},
\begin{align*}
	\sum_{j=-\infty}^\infty 
	(-1)^j t^j q^{j^2}
	&=
	\aqprod{tq,t^{-1}q,q^2}{q^2}{\infty}
.
\end{align*}
We have that
\begin{align*}
	&(1+z)\aqprod{z,z^{-1}}{q}{\infty}S_{F1}(z,q)
	\\	
	&=
		-q^2\sum_{j=-\infty}^\infty (-1)^j z^j q^{j^2-3j}	
		-q\sum_{j=-\infty}^\infty (-1)^j z^j q^{j^2-2j}	
		-zq^2\sum_{j=-\infty}^\infty (-1)^j z^{-j} q^{j^2-3j}
		-zq\sum_{j=-\infty}^\infty (-1)^j z^{-j} q^{j^2-2j}
		\\&\quad
		+(1+z)q^2\sum_{j=-\infty}^\infty (-1)^j q^{j^2-3j}
		+(1+z)\sum_{j=-\infty}^\infty (-1)^j q^{j^2-2j}
	\\
	&=
		-q^2\aqprod{zq^{-2},z^{-1}q^4,q^2}{q^2}{\infty}
		-q\aqprod{zq^{-1},z^{-1}q^3,q^2}{q^2}{\infty}
		-zq^2\aqprod{z^{-1}q^{-2},zq^4,q^2}{q^2}{\infty}
		\\&\quad
		-zq\aqprod{z^{-1}q^{-1},zq^3,q^2}{q^2}{\infty}
		+(1+z)q^2\aqprod{q^{-2},q^4,q^2}{q^2}{\infty}
		+(1+z)q\aqprod{q^{-1},q^3,q^2}{q^2}{\infty}
	\\
	&=
		(1-z^2)\aqprod{z^{-1},zq^2,q^2}{q^2}{\infty}
		+(1+z)q\aqprod{z^{-1}q,zq,q^2}{q^2}{\infty}
		-(1+z)\aqprod{q,q,q^2}{q^2}{\infty}
,
\end{align*}
where the last equality follows from multiple applications of
$\aqprod{t,qt^{-1}}{q}{\infty}=-t\aqprod{tq,t^{-1}}{q}{\infty}$.
We see that (\ref{EqTheoremS3Product}) follows from the above after
diving by $(1+z)\aqprod{z,z^{-1}}{q}{\infty}$ and elementary simplifications.
\end{proof}

\begin{proof}[Proof of (\ref{EqTheoremS4}).]
With $\rho=q$ in (\ref{EqTheoremGInfinity}), we 
have that
\begin{align*}
	S_{G3}(z,q)
	&=
	\frac{1}{(1+z)\aqprod{z,z^{-1}}{q}{\infty}}
	\sum_{j=1}^\infty
	\frac{(1-z^j)(1-z^{j-1})z^{1-j}(1-q^{2j-1})q^{j(j-1)}\aqprod{q}{q}{j-1}
		\aqprod{q^j}{q}{\infty} }
	{q^{j-1}\aqprod{q}{q}{\infty}}
	\\
	&=
	\frac{1}{(1+z)\aqprod{z,z^{-1}}{q}{\infty}}
	\sum_{j=1}^\infty
	(1-z^j)(1-z^{j-1})z^{1-j}(1-q^{2j-1})q^{(j-1)^2}	
,
\end{align*}
which is (\ref{EqTheoremS4}).
\end{proof}

\begin{proof}[Proof of Theorem \ref{TheoremHeckeSeries}.]

The proofs of (\ref{EqTheoremSL7Hecke}), (\ref{EqTheoremSL9Hecke}), and
(\ref{EqTheoremSL12Hecke}) are all a rearrangement of the series in
Theorem \ref{TheoremMain}. We describe the rearrangements here and then proceed
with the calculations. First we reverse the order of summation and expand the 
double series into a sum of two
double series. Second we replace $n$ by $-1-n$ in the second double series.
Third we rewrite the summands in both double series in a common form and
obtain a double series that is bilateral in $n$. Fourth we replace
$j$ by $j-|n|+1$, $j-2|n|+1$, and $j-|n|+1$ for
(\ref{EqTheoremSL7Hecke}), (\ref{EqTheoremSL9Hecke}), and
(\ref{EqTheoremSL12Hecke}) respectively. Lastly we exchange
the order of summation to obtain the identity. Since these calculations are so
similar, we only write out the details for $S_{L7}(z,q)$.

By (\ref{EqTheoremSL7}) we have that
\begin{align*}
	&(1+z)\aqprod{q^2,z,z^{-1}}{q^2}{\infty}S_{L7}(z,q)
	\\
	&=
	\sum_{n=0}^\infty\sum_{j=1}^\infty
		(1-z^j)(1-z^{j-1})z^{1-j} (-1)^{j+1}q^{j(j-1) +\frac{n(n-1)}{2}+2jn}
		\\&\quad	
		+
		\sum_{n=0}^\infty\sum_{j=1}^\infty
			(1-z^j)(1-z^{j-1})z^{1-j} (-1)^{j+1}q^{j(j+1)-1 +\frac{n(n-1)}{2}+2jn}
	\\	
	&=
	\sum_{n=0}^\infty\sum_{j=1}^\infty
		(1-z^j)(1-z^{j-1})z^{1-j} (-1)^{j+1}q^{j(j-1) +\frac{n(n-1)}{2}+2jn}
		\\&\quad
		+
		\sum_{n=-\infty}^{-1}\sum_{j=1}^\infty
		(1-z^j)(1-z^{j-1})z^{1-j} (-1)^{j+1}q^{j(j-1) +\frac{n(n+3)}{2}-2jn}
	\\
	&=
	\sum_{n=-\infty}^{\infty}\sum_{j=1}^\infty
		(1-z^j)(1-z^{j-1})z^{1-j} (-1)^{j+1}q^{j(j-1) +\frac{n(n+1)}{2}-|n|+2j|n|}
	\\
	&=
	\sum_{n=-\infty}^\infty\sum_{j=|n|}^\infty
		(1-z^{j-|n|+1})(1-z^{j-|n|})z^{|n|-j} (-1)^{j+n}q^{j(j+1) -\frac{n(n-1)}{2}}
	\\
	&=
	\sum_{j=0}^\infty\sum_{n=-j}^j
		(1-z^{j-|n|+1})(1-z^{j-|n|})z^{|n|-j} (-1)^{j+n}q^{j(j+1) -\frac{n(n-1)}{2}}
	,
\end{align*}
which implies (\ref{EqTheoremSL7Hecke}).

\end{proof}

\section{Proof of Theorem \ref{TheoremCranks}}

To prove Theorem \ref{TheoremCranks}, we need to show
the coefficients of the following terms are zero:
$q^{3m+1}$ in $S_{L7}(\zeta_3,q)$,
$q^{3m+2}$ in $S_{L9}(\zeta_3,q)$,
$q^{3m}$ in $S_{L12}(\zeta_3,q)$,
$q^{3m}$ in $S_{G1}(\zeta_3,q)$,
$q^{3m}$ in $S_{G2}(\zeta_3,q)$,
$q^{3m+2}$ in $S_{G2}(\zeta_3,q)$,
$q^{10m+9}$ in $S_{F1}(\zeta_5,q)$, and
$q^{5m+3}$ in $S_{G3}(\zeta_5,q)$.

We first note that 
$\aqprod{q,\zeta_3q,\zeta_3^{-1}q}{q}{\infty}=\aqprod{q^3}{q^3}{\infty}$.
By (\ref{EqTheoremSL7}) we have that
\begin{align*}
	S_{L7}(z;q)
	&=
		\frac{-3\zeta_3}{\aqprod{q^6}{q^6}{\infty}}
		\sum_{j=1}^\infty
		\sum_{n=0}^\infty
		(1-\zeta_3^j)(1-\zeta_3^{j-1})\zeta_3^{1-j} (-1)^{j+1}(1+q^{2j-1})q^{j(j-1) +\frac{n(n-1)}{2}+2jn}
.
\end{align*}
We note that the terms in the series are zero except when $j\equiv 2\pmod{3}$. 
However when $j\equiv 2\pmod{3}$, one finds that
$(1+q^{2j-1})q^{j(j-1) +\frac{n(n-1)}{2}+2jn}$
contributes only terms of the form $q^{3m}$ and $q^{3m+2}$. Thus 
$S_{L7}(\zeta_3,q)$ has no non-zero terms of the form $q^{3m+1}$.
Using (\ref{EqTheoremSL9}), (\ref{EqTheoremSL12}), and (\ref{EqTheoremS2})
we find the same reasoning shows the coefficients of
$q^{3m+2}$ in $S_{L9}(\zeta_3,q)$,
$q^{3m}$ in $S_{L12}(\zeta_3,q)$,
$q^{3m}$ in $S_{G2}(\zeta_3,q)$, and
$q^{3m+2}$ in $S_{G2}(\zeta_3,q)$
are all zero.

Next by (\ref{EqTheoremS1}) we have that
\begin{align}
	\label{EqS1CrankStuff}
	S_{G1}(z,q)
	&=
		\frac{-3\zeta_3 \aqprod{q^2}{q^2}{\infty}}{\aqprod{q^6}{q^6}{\infty}\aqprod{q}{q^2}{\infty}}
		\sum_{j=-\infty}^\infty 
		(1-z^j)(1-z^{j-1})z^{1-j}(-1)^{j+1}(1+q^{2j-1})q^{(j-1)^2}
.
\end{align}
By Gauss \cite[Corollary 2.10]{AndrewsBook} we have
\begin{align*}
	\frac{\aqprod{q^2}{q^2}{\infty}}{\aqprod{q}{q^2}{\infty}}
	&=
	\sum_{n=0}^\infty q^{\frac{n(n+1)}{2}}
,
\end{align*}
and so $\frac{\aqprod{q^2}{q^2}{\infty}}{\aqprod{q}{q^2}{\infty}}$ has only terms
of the form $q^{3m}$ and $q^{3m+1}$. In (\ref{EqS1CrankStuff}), the terms in the 
series are zero except when $j\equiv 2\pmod{3}$. 
However when $j\equiv 2\pmod{3}$, one finds that
$(1+q^{2j-1})q^{(j-1)^2}$
contributes only terms of the form $q^{3m+1}$. 
Thus 
$S_{G1}(\zeta_3,q)$ has no non-zero terms of the form $q^{3m}$.

For $S_{F1}(\zeta_5,q)$ and $S_{G3}(\zeta_5,q)$, we first note that
Lemma 3.9 of \cite{Garvan1} is
\begin{align}
	\label{EqGarvan39}
	\frac{1}{\aqprod{\zeta_5q,\zeta_5^{-1}q}{q}{\infty}}
	&=
	\frac{1}{\aqprod{q^5,q^{20}}{q^{25}}{\infty}}
	+	
	\frac{(\zeta_5+\zeta_5^{-1})q}{\aqprod{q^{10},q^{15}}{q^{25}}{\infty}}
	.
\end{align}
Also with the Jacobi triple product identity one can easily deduce that
\begin{align*}
	\aqprod{\zeta_5q^2,\zeta_5^{-1}q^2,q^2}{q^2}{\infty}
	&=
		\aqprod{q^{20},q^{30},q^{50}}{q^{50}}{\infty}	
		+
		(\zeta_5^2+\zeta_5^3)q^2\aqprod{q^{10},q^{40},q^{50}}{q^{50}}{\infty}
	,\\
	\aqprod{\zeta_5q,\zeta_5^{-1}q,q^2}{q^2}{\infty}
	&=
		\aqprod{q^{25},q^{25},q^{50}}{q^{50}}{\infty}	
		-
		(\zeta_5+\zeta_5^4)q\aqprod{q^{15},q^{35},q^{50}}{q^{50}}{\infty}
		\\&\quad		
		+
		(\zeta_5^2+\zeta_5^3)q^4\aqprod{q^{5},q^{45},q^{50}}{q^{50}}{\infty}
	,\\
	\aqprod{q,q,q^2}{q^2}{\infty}
	&=
		\aqprod{q^{25},q^{25},q^{50}}{q^{50}}{\infty}	
		-
		2q\aqprod{q^{15},q^{35},q^{50}}{q^{50}}{\infty}
		+
		2q^4\aqprod{q^{5},q^{45},q^{50}}{q^{50}}{\infty}
	.
\end{align*}
By (\ref{EqTheoremS3Product}) and the above, we have that
\begin{align*}
	S_{F1}(\zeta_5,q)
	&=		
		\frac{\aqprod{\zeta_5^{-1}q^2,\zeta_5q^2,q^2}{q^2}{\infty}}{\aqprod{\zeta_5q,\zeta_5^{-1}q}{q}{\infty}}		
		+
		\frac{\aqprod{\zeta_5^{-1}q,\zeta_5q,q^2}{q^2}{\infty}}{\aqprod{\zeta_5,\zeta_5^{-1}}{q}{\infty}}	
		-
		\frac{\aqprod{q,q,q^2}{q^2}{\infty}}{\aqprod{\zeta_5,\zeta_5^{-1}}{q}{\infty}}	
	\\
	&=	
		\frac{\aqprod{q^{20},q^{30},q^{50}}{q^{50}}{\infty}}
			{\aqprod{q^5,q^{20}}{q^{25}}{\infty}}
		+
		(\zeta_5^2+\zeta_5^3)q^5   
		\frac{\aqprod{q^{5},q^{45},q^{50}}{q^{50}}{\infty}}
			{\aqprod{q^{10},q^{15}}{q^{25}}{\infty}}	
		+
		(\zeta_5+\zeta_5^{4})q 
		\frac{\aqprod{q^{20},q^{30},q^{50}}{q^{50}}{\infty}}
			{\aqprod{q^{10},q^{15}}{q^{25}}{\infty}}
		\\&\quad
		+
		q\frac{\aqprod{q^{15},q^{35},q^{50}}{q^{50}}{\infty}}
			{\aqprod{q^5,q^{20}}{q^{25}}{\infty}}	
		+
		(\zeta_5^2+\zeta_5^3)q^2 
		\frac{\aqprod{q^{10},q^{40},q^{50}}{q^{50}}{\infty}}
			{\aqprod{q^5,q^{20}}{q^{25}}{\infty}}
		+		
		(\zeta_5+\zeta_5^4)q^2 
		\frac{\aqprod{q^{15},q^{35},q^{50}}{q^{50}}{\infty}}
			{\aqprod{q^{10},q^{15}}{q^{25}}{\infty}}	
		\\&\quad
		-
		q^3\frac{\aqprod{q^{10},q^{40},q^{50}}{q^{50}}{\infty}}
			{\aqprod{q^{10},q^{15}}{q^{25}}{\infty}}
		+
		( \zeta_5^3 + \zeta_5^2 -1 )q^4 
		\frac{\aqprod{q^{5},q^{45},q^{50}}{q^{50}}{\infty}}
			{\aqprod{q^5,q^{20}}{q^{25}}{\infty}}	
.
\end{align*}
However, we see that
\begin{align*}
	\frac{\aqprod{q^{5},q^{45},q^{50}}{q^{50}}{\infty}}
		{\aqprod{q^5,q^{20}}{q^{25}}{\infty}}	
	&=
	\frac{\aqprod{q^{5},q^{45},q^{50}}{q^{50}}{\infty}}
		{\aqprod{q^5,q^{20},q^{30},q^{45}}{q^{50}}{\infty}}	
	=
	\frac{\aqprod{q^{50}}{q^{50}}{\infty}}
		{\aqprod{q^{20},q^{30}}{q^{50}}{\infty}}	
,
\end{align*}
and so while $S_{F1}(\zeta_5,q)$ does have terms of the form $q^{5m+4}$, it 
has no terms of the form $q^{10m+9}$.

By (\ref{EqTheoremS4}) we have that
\begin{align*}
	S_{G3}(z,q)
	&=
		\frac{1}{(1+\zeta_5)(1-\zeta_5)(1-\zeta_5^{-1})\aqprod{\zeta q,\zeta^{-1}q}{q}{\infty}}
		\sum_{j=1}^\infty (1-z^j)(1-z^{j-1})z^{1-j}(1-q^{2j-1})q^{(j-1)^2}
.
\end{align*}
However by (\ref{EqGarvan39}), $\frac{1}{\aqprod{\zeta q,\zeta^{-1}q}{q}{\infty}}$
contributes only terms of the form $q^{5n}$ and $q^{5n+1}$. In the series, the
terms are zero except when $j\equiv 2,3,4 \pmod{5}$. One can verify when
$j\equiv 2,3,4 \pmod{5}$ that
$(1-q^{2j-1})q^{(j-1)^2}$ contributes only terms of the form
$q^{5n+1}$ and $q^{5n+4}$. Thus  
$S_{G3}(\zeta_5,q)$ has no non-zero terms of the form $q^{5m+3}$.

\section{Remarks}

Here we have demonstrated that new spt-crank functions arise from conjugate 
Bailey pair identities and variations of Bailey's Lemma, when applied carefully
to the two generic Bailey pairs $(\alpha^*,\beta^*)$ and 
$(\alpha^{**},\beta^{**})$. Previously we saw that spt-crank functions arise
from applying the classical form of Bailey's Lemma to a series
\begin{align*}
	\frac{P(q)}{\aqprod{z,z^{-1}}{q}{\infty}}
	\sum_{n=1}^\infty \aqprod{z,z^{-1}}{q}{n}q^n\beta_n
,
\end{align*}
where $(\alpha,\beta)$ is a Bailey pair relative to $(1,q)$. There are
spt-crank functions that appear in both forms. This was first noticed
in \cite{GarvanJennings2}. In these cases we started with spt-crank functions
defined by the above series and then applied the techniques of this article.
It does appear one can work in the opposite direction as well.
In particular it would appear that $S_{F1}(z,q)$ also comes from an spt-crank function
in the form
\begin{align*}
	\frac{\aqprod{q}{q}{\infty}}{\aqprod{z,z^{-1},-q}{q}{\infty}}
	\sum_{n=1}^\infty \aqprod{z,z^{-1}}{q}{n}q^n\beta_n
,
\end{align*}
with the Bailey pair relative to $(1,q)$ given by
\begin{align*}
	\beta_n 
	&=
	\frac{q^{\frac{n(n-3)}{2}}\aqprod{-q}{q}{n}}{\aqprod{q}{q}{2n}}
	,
	&\alpha_n
	&=
	\left\{ \begin{array}{ll}
		1	& \mbox{ if } n=0
		,\\	
		(-1)^{\frac{n-1}{2}}q^{\frac{n^2-4n-1}{4}}(1-q^{2n})	& \mbox{ if } n \mbox{ is odd}
		,\\
		(-1)^{\frac{n}{2}}q^{\frac{n^2-2n}{4}}(1+q^{n})	& \mbox{ if } n \mbox{ is even}
	.	
	\end{array}\right.	
\end{align*}
With this we could approach $S_{F1}(z,q)$ by applying Bailey's Lemma with
$\rho_1=z$, $\rho_2=z^{-1}$ and obtaining a difference of a generalized Lambert 
series and an infinite product
we could dissect at roots of unity. However we would first need to verify
that the above is indeed a Bailey pair, which we should be able to easily prove
along the same lines as the Bailey pairs from group C of \cite{Slater1}.

The functions $F$ and $G$ are fairly general, so we would expect other 
specializations are of interest as well. Actually many specializations
are functions that were previously studied. To list just a few,
$G(-z^{1/2},-z^{-1/2},z^{1/2};-q)$ is the M2spt crank function $S2(z,q)$
from \cite{GarvanJennings},
$G(q,q,z;q^2)$ is the $\overline{\mbox{spt2}}$ crank function $S(z,q)$ 
from \cite{JS2},
$G(-q^{1/2},q^{1/2},z;q)$ is $S_{E2}(z,q)$ 
and $G(-q,z;q)$ is $S_{C5}(z,q)$
from \cite{GarvanJennings2}.
Additionally, the function $\spt{G2}{n}$ was independently introduced
as $\sptBar{\omega}{n}$ in \cite{AndrewsDixitSchultzYee1}.

\bibliographystyle{abbrv}
\bibliography{someSmallestPartsFromVariationsOfBaileysLemmaRef}

\end{document}